\documentclass[11pt]{article}

\def\UseSection{
        \numberwithin{equation}{section}
        \newtheorem{theorem}    {Theorem}[section]
        \DefineTheorems 
}
\usepackage[textwidth=500pt,textheight=665pt,centering]{geometry} 

\usepackage{enumitem}

\usepackage{amsfonts}
\usepackage{amsmath,amssymb,amsthm}
\usepackage{appendix}
\usepackage{bbm} 
\usepackage{amsbsy}
\usepackage{cite}

\usepackage[bookmarks, colorlinks, breaklinks,
pdftitle={},pdfauthor={}]{hyperref}

\hypersetup{
  linkcolor=black,
  citecolor=black,
  filecolor=black,
  urlcolor=black}

\usepackage{verbatim}
\usepackage{tikz}
\usepackage{color}

\newcommand{\black}{\black}

\usepackage[font=small,labelfont=bf]{caption}

\numberwithin{equation}{section}

\newcommand{\bb}[1]{\mathbb{#1}}

\newcommand{\blank}[1]{}

\newcommand{\R}{\bb R}
\newcommand{\Z}{\bb Z}
\newcommand{\C}{\bb C}

\newcommand{\N}{\bb N}
\newcommand{\T}{\bb T}

\newcommand{\Wcal}   {\mathcal{W}}

\newcommand{\nnb}	{\nonumber \\}

\newcommand{\degree}{\Omega}

\def\DefineTheorems{
	\newtheorem{lemma}      [theorem] {Lemma}

	\newtheorem{prop}        [theorem] {Proposition}
	\theoremstyle{definition}

}

\UseSection
\title{A simple convergence proof for the lace expansion}
\author{Gordon Slade\thanks{Department of Mathematics,
     University of British Columbia,
     Vancouver, BC, Canada V6T 1Z2.
     https://orcid.org/0000-0001-9389-9497.
     E-mail: {{\tt slade@math.ubc.ca}}}}

\begin{document}

\date{\vspace{-5ex}} 

\maketitle

\begin{abstract}
We use the lace expansion to give a simple
proof that the critical two-point function  for weakly self-avoiding
walk on $\mathbb{Z}^d$ has decay $|x|^{-(d-2)}$ in dimensions $d>4$.
The proof uses elementary Fourier analysis and the Riemann--Lebesgue Lemma.
\end{abstract}

%

\section{Introduction and main result}

The lace expansion has been used to prove $|x|^{-(d-2)}$ decay for
the long-distance behaviour of critical two-point
functions in a variety of statistical mechanical models on $\Z^d$
above their upper critical
dimensions, including self-avoiding walk
for $d>4$ \cite{HHS03,Hara08,BHH19,BHK18}, percolation for $d>6$
\cite{HHS03,Hara08}, lattice trees and lattice animals for $d>8$ \cite{HHS03,Hara08},
the Ising model for $d>4$ \cite{Saka07}, and the $\varphi^4$ model
for $d>4$ \cite{Saka15,BHH19}.  For weakly self-avoiding walk and oriented
percolation in dimensions $d>4$, local central limit theorems have
also been proved \cite{BR01,ABR16,HS02}.
Related results for long-range models are proved in \cite{CS15}.

Typically, $|k|^{-2}$ behaviour
for the Fourier transform of the critical two-point function (near $k=0$) had been proved first (as in, e.g, \cite{BS85,Slad06}).  However,
this does not directly imply $|x|^{-(d-2)}$ behaviour for the inverse
Fourier transform;
see \cite[Example~1.6.2]{MS93} for a counterexample and \cite[Appendix~A]{Soka82} for
further discussion of this point.

Our purpose here is to use the lace expansion to give
a simple proof that the critical two-point function
for weakly self-avoiding walk in dimensions $d>4$
has Gaussian decay $|x|^{-(d-2)}$.
Apart from the derivation of the lace expansion, which is well documented in the literature
and not repeated here, our proof uses little more than
elementary Fourier analysis, the Riemann--Lebesgue Lemma, and
the product rule for differentiation.
Although the realm of application of our convergence proof for the lace expansion
appears to be less general than other methods,
its application to weakly self-avoiding walk
is strikingly simple and provides a new tool for problems of this genre.

To make the presentation as simple as possible,
we restrict attention to the
two-point function of the nearest-neighbour weakly self-avoiding walk.
For background we refer to \cite{MS93,Slad06}.
We follow the approach in \cite{HHS03} apart from one key ingredient which is significantly
simplified.

Let $D:\Z^d \to \R$ be given by
$D(x)=\frac{1}{2d}$ if $\|x\|_1=1$ and otherwise $D(x)=0$.  Let $D^{*n}$ denote the $n$-fold
convolution of $D$ with itself.
For $n \in \N$, let $\Wcal_n(x)$ denote the set of $n$-step walks from $0$ to $x$, i.e., the
set of $\omega = (\omega(0),\omega(1),\ldots,\omega(n))$ with each $\omega(i)\in \Z^d$,
$\omega(0)=0$, $\omega(n)=x$,
and $\|\omega(i)-\omega(i-1)\|_1=1$ for $1 \le i \le n$.
The set $\Wcal_0(x)$ consists of the zero-step walk $\omega(0)=0$ when $x=0$,
and otherwise it is the empty set.  We write $\degree = 2d$
for the degree of the nearest-neighbour graph.
The simple random walk two-point function is
defined, for $z \in [0,1/\degree]$, by
\begin{equation}
    C_z(x) = \sum_{n=0}^\infty \sum_{\omega\in \Wcal_n(x)} z^n
    =
    \sum_{n=0}^\infty (z\degree)^n D^{*n}(x).
\end{equation}
The Green function is $C_{1/\degree}(x)$.

For $\omega \in \Wcal_n(x)$ and $0 \le s < t \le n$, we define
\begin{equation}
    U_{st}(\omega) =
    \begin{cases}
        -1 & (\omega(s)=\omega(t))
        \\
        0 & (\text{otherwise}).
    \end{cases}
\end{equation}
Given $\beta \in (0,1)$, $z \ge 0$, and $x \in \Z^d$,
the \emph{weakly self-avoiding walk two-point function} is then
defined by
\begin{equation}
    G_z(x) = \sum_{n=0}^\infty \sum_{\omega\in \Wcal_n(x)} z^n \prod_{0\le s<t\le n}(1+\beta U_{st}(\omega)).
\end{equation}
The \emph{susceptibility} is defined by $\chi(z)=\sum_{x\in \Z^d} G_z(x)$.
A standard subadditivity argument implies the existence of $z_c=z_c(\beta) \ge
z_c(0)=1/\degree$ such that
$\chi(z)$ is finite if and only if $z \in [0,z_c)$; also $\chi(z) \ge z_c/(z_c-z)$
so $\chi(z_c)=\infty$ (see, e.g., \cite[Theorem~2.3]{Slad06}).
In particular, $G_z(x)$ is finite if $z\in [0,z_c)$; in fact it decays exponentially in $x$.
We will prove the following theorem.  The constant $a_d$ in the theorem is
$a_d=\frac{d\Gamma(\frac{d-2}{2})}{2 \pi^{d/2}}$.

\begin{theorem}
\label{thm:wsaw}
Let $d>4$, and let $\beta>0$ be sufficiently small.  There is a constant
$c_d = a_d (1 +O(\beta))$ such that
\begin{equation}
    G_{z_c}(x) = c_d \frac{1}{|x|^{d-2}} + o\left( \frac{1}{|x|^{d-2}} \right).
\end{equation}
For $d > 5$, the error term is improved to $o(|x|^{-(d-1)})$.
\end{theorem}

In \cite{Slad20_wsaw}, the method of proof of Theorem~\ref{thm:wsaw} is extended to
analyse the near-critical two-point function.  Namely, it is proved in \cite{Slad20_wsaw}
that for $d>4$ and for $\beta$ sufficiently small, there are constants $\kappa_0>0$
and $\kappa_1 \in (0,1)$ such that for all $z \in (0,z_c)$ and all $x \in \Z^d$,
\begin{equation}
\label{e:Gmr}
    G_z(x) \le \kappa_0 \frac{1}{1\vee |x|^{d-2}} e^{-\kappa_1 (z_c-z)^{1/2}|x|}.
\end{equation}
The estimate \eqref{e:Gmr} is applied in \cite{Slad20_wsaw} to prove existence of
a ``plateau'' for the weakly self-avoiding walk two-point function on a large discrete torus
in dimensions $d>4$.

An alternate proof of Theorem~\ref{thm:wsaw} is given in \cite{BHK18}, based on
Banach algebras and a fixed-point theorem.  That proof avoids explicit use of the Fourier transform, though
it does rely on an expansion of the Green function
$C_{1/\degree}(x)$ which is proved using the Fourier transform.
Theorem~\ref{thm:wsaw} is proved for the strictly self-avoiding walk (the case $\beta =1$)
in \cite{Hara08}, and for spread-out strictly self-avoiding walk in \cite{HHS03}.
Thus Theorem~\ref{thm:wsaw} is not new or best possible; our goal here is to present
a new and simple method of proof rather than to obtain a new result.
A sample consequence of Theorem~\ref{thm:wsaw} is that the bubble condition holds for $d>4$,
and this implies the matching upper bound $\chi(z) \le O((z_c-z)^{-1})$ (see \cite[Theorem~2.3]{Slad06}).

We use the Fourier transform.
Let $\mathbb{T}^d = (\R / 2 \pi \Z )^d$ denote the continuum torus of period $2\pi$.
For a summable function $f:\Z^d \to \C$ we define its Fourier transform by
\begin{equation}
    \hat f(k) = \sum_{x\in\Z^d}f(x)e^{ik\cdot x}
    \qquad
    (k \in \T^d).
\end{equation}
The inverse Fourier transform is
\begin{equation}
    f(x) = \int_{\T^d}\hat{f}(k) e^{-i k \cdot x } \frac{dk}{(2\pi)^{d}}
    \qquad
    (x \in \Z^d).
\end{equation}

\section{Lace expansion}

The lace expansion was introduced by Brydges and Spencer \cite{BS85} to prove
that the weakly self-avoiding walk is diffusive in dimensions $d>4$.
In the decades since 1985, the lace expansion has been adapted and extended to a broad range of
models and results.

For the weakly self-avoiding walk, the lace expansion
\cite{BS85,MS93,Slad06}
produces an explicit formula for the $\Z^d$-symmetric function $\Pi_z: \Z^d \to \R$
which satisfies, for $z\in [0,z_c)$,
\begin{equation}
    G_z(x) = \delta_{0,x} + z\degree (D*G_z)(x) + (\Pi_z* G_z)(x) \qquad (x \in \Z^d),
\end{equation}
or equivalently,
\begin{equation}
    \hat G_z(k) = \frac{1}{1-z\degree \hat D(k) - \hat \Pi_z(k)}
    \qquad (k \in \T^d).
\end{equation}
Let $\delta:\Z^d\to\R$ denote the Kronecker delta $\delta(x)=\delta_{0,x}$.
Then $\hat\delta(k)=1$.
We define
\begin{equation}
    F_z  = \delta - z\degree D  - \Pi_z , \qquad \hat F_z = 1 -z\degree\hat D -\hat\Pi_z.
\end{equation}
Then
\begin{equation}
     (G_z * F_z)(x)  = \delta_{0,x} , \qquad \hat G_z(k) = \frac{1}{\hat F_z(k)}.
\end{equation}

\section{Proof of main result}

\subsection{Diagrammatic estimate}

As in many applications of the lace expansion, we use a bootstrap argument.
We define the \emph{bootstrap function}
\begin{equation}
    b(z) = \sup_{x\in \Z^d} \frac{G_z(x)}{C_{1/\degree}(x)} \qquad (z \in [0,z_c]).
\end{equation}
The bootstrap function can be seen to be finite and
continuous in $z\in [0,z_c)$, using the fact that $G_z(x)$ is continuous and decays exponentially for large $|x|$.  We do not know {\it a priori} that $G_{z_c}(x)$ is finite.
By definition, $b(z) \le 1$ for $z \in [0,\frac{1}{\degree}]$.
The next proposition gives consequences of the assumption that $b(z) \le 3$.
We will not need to know more about the function $\Pi_z$ than
Proposition~\ref{prop:diagram}, so we do not give
its explicit formula here.  The formula can be found in \cite{BS85,MS93,Slad06}.

\begin{prop}
\label{prop:diagram}
Let $d>4$ and let $\beta$ be sufficiently small.  Fix $z \in [\frac{1}{\degree},z_c]$.
If $b(z) \le 3$ then there is a constant $K$ depending only on $d$ (and on ``3'') such that
\begin{equation}
\label{e:Pibd}
    |\Pi_z(x)| \le K\beta \frac{1}{1+|x|^{3(d-2)}}
    \qquad
    (x \in \Z^d)
\end{equation}
and hence $\hat\Pi_z \in C^{s}(\T^d)$
for any nonnegative integer $s< 2d-6$, in particular $\hat\Pi_z \in C^{d-2}(\T^d)$.
In addition, the \emph{infrared bound} holds, i.e., there exists $c>0$ (independent
of $\beta,z,k$) such that
\begin{equation}
\label{e:irbd}
    \hat F_z(k) \ge  c  |k|^{2}
    \qquad (k \in \T^d).
\end{equation}
\end{prop}

\begin{proof}
The bound \eqref{e:Pibd} is a diagrammatic estimate proved via well-developed
technology (e.g., \cite[(12)]{BHK18})
and we omit its proof.  It follows immediately from \eqref{e:Pibd}
that $|x|^s|\Pi_z(x)|$ is summable for all $s< 2d-6$ and hence
that $\hat\Pi_z \in C^{s}(\T^d)$
for any nonnegative integer $s< 2d-6$.
For \eqref{e:irbd}, we use
\begin{align}
    \hat F_z(k) & = \hat F_z(0) +[\hat F_z(k) -\hat F_z(0)]
    = \hat F_z(0) + z\degree [1-\hat D(k)] + [\hat\Pi_z(0)-\hat\Pi_z(k)].
\end{align}
The first term on the right-hand side is  $\hat F_z(0) = \chi(z)^{-1} \ge 0$.
By definition, the second term obeys $1-\hat D(k) = d^{-1}\sum_{j=1}^d(1-\cos k_j) \ge \frac{4|k|^2}{\pi^2\degree}$.
By \eqref{e:Pibd}, the second derivative of $\hat\Pi_z(k)$ with respect to $k$ is $O(\beta)$,
and \eqref{e:irbd} then follows by a Taylor estimate on $\hat\Pi_z(0)-\hat\Pi_z(k)$
(by symmetry there is no linear term in $k$).
\end{proof}

\subsection{Isolation of leading term}

Following \cite{HHS03}, we isolate the leading term by writing $G_z$ as a $z$-dependent
multiple of the random walk two-point function $C_\mu$ at a $z$-dependent value of $\mu$.
Let  $A_z = \delta -z\degree D$, $\lambda  > 0$ and $\mu\in[0,\frac{1}{\degree}]$.
Since $C_\mu * A_\mu=\delta$ and $G_z *F_z = \delta$,
we have
\begin{align}
    G_z & = \lambda C_\mu + \delta*G_z - \lambda C_\mu * \delta
    \nnb & =
    \lambda C_\mu + C_\mu * A_\mu*G_z - \lambda C_\mu  *F_z * G_z
    \nnb & =
    \lambda C_\mu + C_\mu * E_{z,\lambda,\mu}*G_z
    ,
\end{align}
with
\begin{align}
     E_{z,\lambda,\mu} &  = A_\mu - \lambda F_z.
\end{align}

Given $z \in [\frac{1}{\degree},z_c)$,
we choose $\lambda=\lambda_z$ and $\mu=\mu_z$ in order to achieve
\begin{equation}
\label{e:Emoments}
    \sum_{x\in\Z^d} E_{z,\lambda_z,\mu_z}(x) = \sum_{x\in\Z^d} |x|^2 E_{z,\lambda_z,\mu_z}(x) =0,
\end{equation}
namely (since $\sum_x |x|^2 D(x)=1$)
\begin{align}
    \lambda_z & = \frac{1}{\hat F_z(0) -
    \sum_x |x|^2 F_z(x)}
    =
    \frac{1}{1-\hat \Pi_z(0) +
    \sum_x |x|^2 \Pi_z(x)} ,
    \\
\label{e:muzdef}
    \mu_z \degree & = 1-\lambda_z \hat F_z(0)
        = \frac{z\degree + \sum_x |x|^2 \Pi_z(x)}
    {\hat F_z(0) + z\degree + \sum_x |x|^2 \Pi_z(x)}
    .
\end{align}
By Proposition~\ref{prop:diagram}, if we assume $b(z) \le 3$
then the second moment of $\Pi_z$ is $O(\beta)$, and hence
the above formulas are well defined, $\lambda_{z} = 1+O(\beta)$, and $\mu_z\degree \in [0,1)$.
In particular, if $b(z_c)\le 3$
(as we will eventually show to be the case), then, since $\hat F_{z_c}(0)=\chi(z_c)^{-1}=0$,
we see from \eqref{e:muzdef} that $\mu_{z_c} =1/\degree$ is the critical
value for $C_\mu$.

With these choices of $\lambda_z,\mu_z$, we have
\begin{align}
\label{e:GSEG}
    G_z & =
    \lambda_z C_{\mu_z} + f_z, \qquad f_z = C_{\mu_z} * E_{z}*G_z
    ,
\end{align}
with
\begin{equation}
\label{e:Ezdef}
    E_z = E_{z,\lambda_z,\mu_z} =
    (1-\lambda_z )(\delta-D) - \lambda_z \hat \Pi_z(0) D + \lambda_z \Pi_z.
\end{equation}
By definition,
\begin{equation}
\label{e:fhdef}
    \hat f_z(k) =  \hat C_{\mu_z}(k) \hat E_z(k) \hat G_z(k).
\end{equation}
Roughly, since we have arranged
via \eqref{e:Emoments}
that the Taylor expansion of $\hat E_z(k)$ has no constant
term or term of order $|k|^2$, we expect it to be of order $\beta |k|^4$.
On the other hand, according to the
infrared bound, the Fourier transform of $\hat G_z(k)$
will be of order $|k|^{-2}$ for small $|k|$.  The same is true for $\hat C_{\mu_z}(k)$,
so $\hat f_z(0)=O(\beta)$.
We will show that
this less singular behaviour of $\hat f_z(k)$
translates into better decay than $|x|^{-(d-2)}$ for $f_z(x)$.
This will permit
the bootstrap argument to be completed by proving $b(z) \le 2$, and the proof will
essentially be complete.  The details in this rough sketch are given below.

\subsection{The bootstrap}

The bootstrap argument is encapsulated in the following proposition.

\begin{prop}
\label{prop:boot}
Fix $z \in [\degree^{-1},z_c)$.
If $b(z) \le 3$ then for $\beta$ sufficiently small (not depending on $z$) it is in fact
the case that $b(z) \le 2$.
\end{prop}

The next proposition is a replacement for the bound
on $ C_{\mu_z}*  E_z$ in \cite[Proposition~1.9]{HHS03} which required a delicate
Fourier analysis of $C_{\mu_z}$, and of the bound of \cite[Lemma~4]{BHK18} which used
the expansion $ C_{1/\degree}(x) = a|x|^{-(d-2)} + b|x|^{-d} + O(|x|^{-(d+2)})$ from \cite{Uchi98}
which was also proved by careful Fourier analysis.

\begin{prop}
\label{prop:h}
Let $d>4$ and let $\beta$ be sufficiently small.
Let $z\in [\frac{1}{\degree},z_c]$.
Under the assumption that $b(z) \le 3$, the derivatives $\nabla^\alpha \hat f_z(k)$ obey
\begin{equation}
\label{e:nablafbd}
    \|\nabla^\alpha \hat f_z  \|_{L^1(\T^d )} \le O(\beta)
\end{equation}
provided $|\alpha|\le d-2$ when $d=5$ and $|\alpha| \le d-1$ for $d \ge 6$,
with constant depending only on $d$ (not on $z$).
\end{prop}

The importance of \eqref{e:nablafbd} resides in the fact that the smoothness
of a function on the torus implies bounds on the decay of its (inverse) Fourier transform.
More precisely, it is proved in \cite[Corollary~3.3.10]{Graf10} using integration by parts that
there is a constant $\kappa_{d,s}$, depending only on the dimension $d$ and the
maximal order $s$
of differentiation,
such that if $\nabla^\alpha \hat\psi
\in L^1(\T^d)$ for all multi-indices $\alpha$ with $|\alpha| \le s$ then
\begin{equation}
\label{e:Graf}
    |\psi(x)|
    \le
    \kappa_{d,s} \frac{1}{1\vee |x|^s}
    \max_{|\alpha| \in \{0,s\}}\|\nabla^\alpha \hat\psi\|_{L^1(\T^d)}.
\end{equation}
In addition to the quantitative estimate \eqref{e:Graf}, it follows from
the integrability of
$\nabla^\alpha \hat\psi$ together with the Riemann--Lebesgue Lemma (e.g.,
\cite[Proposition~3.3.1]{Graf10}) that the inverse Fourier transform of
$\nabla^\alpha \hat\psi$ vanishes at infinity, so
$|x|^s \psi(x) \to 0$ as $|x|\to\infty$.

\begin{proof}[Proof of Proposition~\ref{prop:boot}]
Suppose that $b(z) \le 3$.
Let $n_d=d-2$ for $d=5$ and $n_d=d-1$ for $d \ge 6$.
By \eqref{e:nablafbd}, as discussed around \eqref{e:Graf},
$|f_z(x)|= o(|x|^{-n_d})$ and
$|f_z(x)| \le  O(\beta |x|^{-n_d})$,
with $z$-independent constant in the latter bound.
Therefore,
\begin{align}
\label{e:GCasy}
    G_z(x) & = \lambda_z C_{\mu_z}(x)  + o(|x|^{-n_d}),
    \\
\label{e:GCbd}
    G_z(x) & = \lambda_z C_{\mu_z}(x)  + O(\beta |x|^{-n_d}),
\end{align}
with the constant in \eqref{e:GCbd} independent of $x$ and $z$.
As is well known (e.g., \cite{Uchi98,LL12}),
$C_{1/\degree}(x) \sim a_d  |x|^{-(d-2)}$.
From \eqref{e:GCbd},
we see that by taking $\beta$ sufficiently small we can obtain
\begin{equation}
    G_z(x) \le (1+O(\beta))C_{1/\degree}(x) + O(\beta)C_{1/\degree}(x) \le 2 C_{1/\degree}(x),
\end{equation}
i.e., $b(z) \le 2$.
This completes the proof.
\end{proof}

\begin{proof}[Proof of Theorem~\ref{thm:wsaw}]
Since $b(z) \le b(\degree^{-1})\le 1$ for $z \le \degree^{-1}$ by definition,
it follows from Proposition~\ref{prop:boot} and the continuity of the function $b$
that the interval $(2,3]$ is forbidden for values of $b(z)$,
so $b(z) \le 2$ for all $z\in [0,z_c)$.  By monotone convergence, also $b(z_c) \le 2$.  Thus $\lambda_z$ approaches a limit
$\lambda_{z_c}=1+O(\beta)$ as $z \to z_c^-$.
Since $b(z_c) \le 2$, \eqref{e:GCasy} holds for $z=z_c$ so
$G_{z_c}(x)=\lambda_{z_c}C_{1/\degree}(x) +o(|x|^{-n_d})$.
Thus Theorem~\ref{thm:wsaw} is proved subject to Proposition~\ref{prop:h}.
\end{proof}

\subsection{Proof of Proposition~\ref{prop:h}}

It remains only to prove Proposition~\ref{prop:h}.
The next lemma is closely related to \cite[Lemma~7.2]{HHS03}.  It
reflects our choice of $\lambda_z,\mu_z$ to achieve \eqref{e:Emoments}.

\begin{lemma}
\label{lem:Ebd}
Let $d>4$ and let $\beta$ be sufficiently small.
Let $z \in [\frac{1}{\degree},z_c]$ and suppose that $b(z) \le 3$.
There is a $c_0>0$ (independent of $z$)
such that $|\nabla^\alpha \hat E_z(k)| \le c_0 \beta$
for $|\alpha| < 2d-6$, and moreover
\begin{equation}
    |\nabla^\alpha \hat E_z(k)| \le
    c_0  \beta \times
    \begin{cases}
    |k|^{4-|\alpha|} & (d>5)
    \\
    |k|^{4-|\alpha|} \log |k|^{-1} & (d=5)
    \end{cases}
    \qquad (|\alpha| \le 3).
\end{equation}
\end{lemma}

\begin{proof}
We first prove that $|\nabla^\alpha \hat E_z(k)| \le c_0 \beta$
for $|\alpha| < 2d-6$, via
the inequality $|\nabla^\alpha \hat E_z(k)| \le \sum_x |x^\alpha E_z(x)|$ together with
term-by-term estimation in the formula for $E_z$ in \eqref{e:Ezdef}.
Indeed, since $\lambda_z=1+O(\beta)$, the contribution from all moments of the term
$(1-\lambda_z)(\delta - D)$ is $O(\beta)$, and since $|\hat\Pi_z(0)| \le O(\beta)$
by \eqref{e:Pibd}, all moments of $\lambda_z \hat\Pi_z(0)D$ are also $O(\beta)$.
Finally, the moments of $\lambda_z|\Pi_z|$ with order less than $2d-6$
are bounded by $O(\beta)$ using \eqref{e:Pibd}.
In the remainder of the proof, we can therefore restrict to $|\alpha|\le 3$.

Let $g_x(k) = \cos (k\cdot x) - 1 + \frac{(k\cdot x)^2}{2!}$.
By symmetry and by \eqref{e:Emoments},
\begin{align}
    \hat E_z(k) & = \sum_{x\in\Z^d}  E_z(x) \cos (k\cdot x)
    =  \sum_{x\in\Z^d}  E_z(x) g_x(k)
    .
\end{align}
By explicit computation of the derivatives and elementary properties of sine and cosine,
\begin{gather}
     |g_x(k)| \le c(|k|^4|x|^4 \wedge (1 + |k|^2|x|^2)),
    \qquad
    |\nabla_i g_x(k)| \le c(|k|^3|x|^4 \wedge (|x|+|k||x|^2)),
    \\
     |\nabla_{ij}^2 g_x(k)| \le c(|k|^2|x|^4 \wedge |x|^2),
    \qquad
    |\nabla_{ijl}^3 g_x(k)| \le c(|k||x|^4 \wedge |x|^3),
    \qquad
\end{gather}
where $\wedge$ denotes minimum.  In the upper bound $|\nabla^\alpha \hat E_z(k)|
\le \sum_x |E_z(x)|\, |\nabla^{\alpha}g_x(k)|$,
we estimate the sum over $|x|\le |k|^{-1}$ using the $|x|^4$ bound on $\nabla^\alpha g_x(k)$,
and we estimate the sum over $|x|> |k|^{-1}$ using the other alternative in the minimum.
By a term-by-term estimate using \eqref{e:Ezdef} with \eqref{e:Pibd} (as in the previous
paragraph), for $|\alpha| \le 3$ and $d \ge 5$, and for small $|k|$,
\begin{align}
    \sum_{|x| \le |k|^{-1}} |x|^4 |E_z(x)|
    & \le
    O(\beta) \int_1^{|k|^{-1}} \frac{r^{d-1+4}}{r^{3(d-2)}}dr
    =
    O(\beta) \int_1^{|k|^{-1}} \frac{1}{r^{2d-9}}dr
    =
    \begin{cases}
        O(\beta) & (d>5)
        \\
        O(\beta \log |k|^{-1}) & (d=5),
    \end{cases}
\\
    \sum_{|x| > |k|^{-1}} |x|^{|\alpha|} |E_z(x)|
    & \le
    O(\beta) \int_{|k|^{-1}}^\infty \frac{r^{d-1+|\alpha|}}{r^{3(d-2)}}dr
    =
    O(\beta) \int_{|k|^{-1}}^\infty\frac{1}{r^{2d-5-|\alpha|}}dr
    =
        O(\beta |k|^{2d-6-|\alpha|}) ,
\end{align}
and the desired result then follows after some bookkeeping.
\end{proof}

\begin{proof}[Proof of Proposition~\ref{prop:h}]
Let $n_d=d-2$ for $d=5$ and $n_d=d-1$ for $d \ge 6$; then $n_d < 2d-6$.
Our goal is to prove the bound \eqref{e:nablafbd} on derivatives of $\hat f_z$ of order
up to $n_d$.
By the infrared bound and Lemma~\ref{lem:Ebd},  for $d>5$ we have
\begin{equation}
    |\hat f_z(k)| \le  |k|^{-2} O( \beta) |k|^4 |k|^{-2} = O(\beta),
\end{equation}
and a similar estimate holds for $d=5$ with an additional factor $\log |k|^{-1}$ in
the upper bound.

Estimation of the $L^1$ norm of derivatives of $\hat f_z(k)$
is an exercise in power counting, as follows.
For $|\alpha| \le n_d$, by the product rule for differentiation
$\nabla^\alpha \hat f_z(k)$ involves terms
\begin{equation}
    \nabla^{\alpha_1}\hat C_{\mu_z}(k)
    \nabla^{\alpha_2}\hat E_{z}(k) \nabla^{\alpha_3}\hat G_{z}(k) ,
    \qquad
    (|\alpha_1|+|\alpha_2|+|\alpha_3| =|\alpha|).
\end{equation}
By Lemma~\ref{lem:Ebd},
each derivative on $\hat E$ up to the fourth order reduces by $1$ the original power $|k|^4$ for that factor (an unimportant factor $\log |k|^{-1}$ is present for $d=5$),
and subsequent derivatives do not cause further reduction; the net effect
is therefore reduction by $\min\{|\alpha_2|,4\}$.
Similarly, each derivative on $\hat C_{\mu_z}$ or $\hat G_z$
reduces (worsens)
its power $|k|^{-2}$ by $1$; we illustrate the idea for $d=5$, for which $n_d=d-2=3$:
\begin{align}
    \left| \nabla_i \frac{1}{ \hat F} \right|
    &
    = \left| \frac{\nabla_i\hat F }{\hat F^2} \right|
    \le c_1 \frac{|k_i|}{ |k|^{4}} \le c_1 \frac{1}{ |k|^{3}},
    \qquad
    \left| \nabla_i^2 \frac{1}{\hat F} \right|
    \le \left| \frac{\nabla_i^2\hat F }{\hat F^2} \right|
    +
    \left| \frac{2( \nabla_i\hat F)^2 }{ \hat F^3} \right|
    \le c_2 \frac{k_i^0}{|k|^4} + c_2 \frac{k_i^2}{|k|^6} ,
     \\
    \left| \nabla_i^3 \frac{1}{\hat F} \right|
    &
    \le \left| \frac{\nabla_i^3\hat F }{\hat F^2} \right|
    +
    \left| \frac{6( \nabla_i\hat F)(\nabla_i^2 \hat F) }{ \hat F^3} \right|
    +
    \left| \frac{6( \nabla_i\hat F)^3 }{ \hat F^4} \right|
    \le c_3 \frac{k_i^0}{|k|^4} + c_3 \frac{k_i^{1+0}}{|k|^6} +c_3 \frac{k_i^3}{|k|^8}   .
\end{align}
A detail in the above calculation is that $|\nabla_i\hat F|$ can be bounded by a multiple
of $|\nabla_i\hat D| + |\nabla_i\hat \Pi|$, with the first term of order $k_i$ by explicit
calculation and the second also of order $k_i$ by Taylor's Theorem, symmetry,
and the boundedness of $\nabla_i^2\Pi$.

In the general case, in advancing from one derivative to the next, when the derivative acts
on the numerator it either maintains the same power of $|k|$ or reduces it
by $1$, and if it acts on the denominator then it increases the power of the denominator
by $|k|^2$ and increases the power of the numerator by $1$; the net result is
reduction of the overall power of $|k|$ by at most $1$.
For $|\alpha| \le n_d$, the total resulting power is (for small $|k|$) at worst
\begin{equation}
    \frac{|k|^{4-\min\{|\alpha_2|,4\}}}{|k|^{2+ |\alpha_1|} |k|^{2+|\alpha_3|}}
    =
    \frac{1}{|k|^{ |\alpha_1|+\min\{|\alpha_2|,4\} +|\alpha_3|}}
    \le
    \frac{1}{|k|^{|\alpha|}}
    \le
    \frac{1}{|k|^{n_d}}.
\end{equation}
This is in $L^1(\T^d)$ (also with the additional logarithmic factor when $d=5$)
and the norm in $L^1(\T^d)$
is $O(\beta)$ due to the factor $\beta$ in the bound on $\hat E_z$ in Lemma~\ref{lem:Ebd}.
This completes the proof.\footnote{An earlier proof of Proposition~\ref{prop:h}
was based on Kotani's Theorem \cite{Slad20_Kotani}
rather than via the direct proof given here.}
\end{proof}

\subsection{Concluding remarks}

\begin{enumerate}[label=(\roman*)]
\item
The main difference between the above proof and the proofs in \cite{HHS03,Hara08,BHK18}
is our avoidance of any need to convert the decay in $x$-space of factors in a convolution
$C_{\mu_z}*E_z*G_z$ into decay of the convolution (which is delicate
since, e.g., the convolution of two factors each with decay $|x|^{-(d-2)}$ has worse
decay $|x|^{-(d-4)}$ when $d>4$).  In the above proof, we encounter
instead the Fourier transform $\hat C_{\mu_z}\hat E_z \hat G_z$ and the corresponding
step is handled simply via the product rule for differentiation.

\item
To control the lace expansion,
the analysis in any of \cite{HHS03,Hara08,BHK18}
only requires a bound of the form $|\Pi(x)|\le O(|x|^{-(d+2+\epsilon)})$ for $\epsilon>0$
(with $\epsilon=2$ in \cite{BHK18}), whereas
the above proof
requires the more demanding bound $|\Pi(x)| \le  O(|x|^{-(2d-2+\epsilon)})$
in order for $\hat\Pi(k)$ to have derivatives of order $d-2$.
For self-avoiding walk, the upper bound \eqref{e:Pibd} has power $3(d-2) = 2d-2 + (d-4)$ which is sufficient.
For the Ising and $\varphi^4$ models, $\Pi(x)$  also
obeys an upper bound $|x|^{-3(d-2)}$ \cite{BHH19,Saka07,Saka15}.  However,
the above proof appears not to apply to percolation or to lattice trees and
lattice animals, where
the bound on $\Pi(x)$ is $|x|^{-2(d-2)}$ for percolation and $|x|^{-3(d-2)+d}$ for
lattice trees and lattice animals \cite{HHS03}.  It would be of interest to understand
better why this breakdown occurs and whether there is any possiblity to overcome it in these settings with upper critical dimension equal to $6$ or $8$.

\item
With further effort, it may be possible to extend our approach to spread-out
models of strictly self-avoiding walk or the Ising or $\varphi^4$ models in dimensions $d>4$
by proving a version of \eqref{e:nablafbd} in those settings.
Possibly this could simplify aspects of the analysis in \cite{HHS03,Saka07,Saka15,BHH19}.
However, this is beyond our current scope and we do not draw a conclusion about
this question here.

\item
It is natural to ask whether our approach could be used for the nearest-neighbour
strictly self-avoiding walk in dimensions $d \ge 5$, to give an alternate proof
of the $|x|^{-(d-2)}$ decay proved in \cite{Hara08}.
This certainly could not be done, at present, without
a portion of the very sizeable and computer-assisted input from \cite{HS92b} listed
in \cite[Proposition~1.3]{Hara08}; in fact results beyond those of \cite{HS92b} may be needed
to deal with the higher derivatives of $\hat\Pi(k)$ encountered here.
A further and serious obstacle, as pointed out and overcome
with a different method in \cite{Hara08}, is that
the amplitude $a_d = \frac d2 \pi^{-d/2} \Gamma((d-2)/2)$ in the asymptotic formula for the critical simple random walk two-point
function (and appearing in Theorem~\ref{thm:wsaw})
grows rapidly with $d$, so the small parameter that facilitates the bootstrap argument is
more hidden and more delicate to exploit.
\end{enumerate}

\section*{Acknowledgement}

This work was supported in part by NSERC of Canada.


\end{document}